\documentclass[a4paper]{amsart}

\usepackage{amsthm}
\usepackage{courier}
\usepackage[all,cmtip]{xy}
\usepackage{pb-diagram, pb-xy}

\newtheorem{theorem}{Theorem}
\newtheorem{corollary}{Corollary}
\newtheorem{lemma}{Lemma}
\newtheorem{proposition}{Proposition}
\newtheorem{definition}{Definition}

\begin{document}
\title{Totally acyclic complexes}

\author[S. Estrada]{Sergio Estrada}

\address{S.E. \ Universidad de Murcia, Murcia 30100, Spain}

\email{sestrada@um.es}

\author[X. Fu]{Xianhui Fu}

\address{X.F.\ School of Mathematics and Statistics, Northeast Normal University, Changchun, China}

\email{fuxianhui@gmail.com}

\author[A. Iacob]{Alina Iacob}

\address{A.I. \ Georgia Southern University, Statesboro, GA 30460,
  U.S.A.}

\email{aiacob@georgiasouthern.edu}


\thanks{Acknowledgment. Part of this research was conducted while S. Estrada and A. Iacob visited the Banff International Research Station for Mathematical Innovation and Discovery (BIRS) and the Mathematiches Forschunginstitut Oberwolfach (MFO). The hospitality and support of both institutions is acknowledged with gratitude. Part of the research was carried out when the first author visited Northeast Normal University in China, he gratefully acknowledges the hospitality of the institution. The second author was partially supported by the National Natural Science Foundation of China, Grant No. 1301062. The first author was supported by the grant MTM2013-46837-P and FEDER funds. The first and third authors are also supported by the grant 18394/JLI/13 by the Fundaci\'on S\'eneca-Agencia de Ciencia y Tecnolog\'{\i}a de la Regi\'on de Murcia in the framework of III PCTRM 2011-2014.
}

\begin{abstract}
We prove first (Proposition 3) that, over any ring $R$, an acyclic complex of projective modules is totally acyclic if and only if the cycles of every acyclic complex of Gorenstein projective modules are Gorenstein projective. The dual result for injective and Gorenstein injective modules also holds over any ring $R$ (Proposition 4). And, when $R$ is a GF-closed ring, the analogue result for flat/Gorenstein flat modules is also true (Proposition 5). Then we show (Theorem 2) that over a left noetherian ring $R$, a third equivalent condition can be added to those in Proposition 4, more precisely, we prove that the following are equivalent:
1. Every acyclic complex of injective modules is totally acyclic.
2. The cycles of every acyclic complex of Gorenstein injective modules are Gorenstein injective.
3. Every complex of Gorenstein injective modules is dg-Gorenstein injective.
Theorem 3 shows that the analogue result for complexes of flat and Gorenstein flat modules holds over any left coherent ring $R$.
We prove (Corollary 1) that, over a commutative noetherian ring $R$, the equivalent statements in Theorem 3 hold if and only if the ring is Gorenstein. 
We also prove (Theorem 4) that when moreover $R$ is left coherent and right $n$-perfect (that is, every flat right $R$-module has finite projective dimension $\leq n$)  then statements 1, 2, 3 in Theorem 2 are also equivalent to the following: 4. Every acyclic complex of projective right $R$-modules is totally acyclic. 5. Every acyclic complex of Gorenstein projective right $R$-modules is in $\widetilde{\mathcal{GP}}$. 6. Every complex of Gorenstein projective right $R$-modules is dg-Gorenstein projective.\\
Corollary 2 shows that when $R$ is commutative noetherian of finite Krull dimension, the equivalent conditions (1)-(6) from Theorem 4 are also equivalent to those in Theorem 3 and hold if and only if $R$ is an Iwanaga-Gorenstein ring. Thus we improve slightly on a result of Iyengar's and Krause's; in \cite{IyengarKrause} they proved that for a commutative noetherian ring $R$ with a dualizing complex, the class of acyclic complexes of injectives coincides with that of totally acyclic complexes of injectives if and only if $R$ is Gorenstein. We are able to remove the dualizing complex hypothesis and add more equivalent conditions.\\
In the second part of the paper we focus on two sided noetherian rings that satisfy the Auslander condition. We prove (Theorem 7) that for such a ring $R$ that also has finite finitistic flat dimension, every complex of injective (left and respectively right) $R$-modules is totally acyclic if and only if $R$ is an Iwanaga-Gorenstein ring.
\end{abstract}
\maketitle

\section{introduction}

Homological algebra is at the root of modern techniques in many areas of mathematics including commutative and non commutative algebra, algebraic geometry, algebraic topology and representation theory. Not only that all these areas make use of the homological methods but homological algebra serves as a common language and this makes interactions between these areas possible and fruitful. A relative version of homological algebra is the area called Gorenstein homological algebra. This newer area started in the late 60’s when Auslander introduced a class of finitely generated modules that have a complete resolution. Auslander used these modules to define the notion of the G-dimension of a finite module over a commutative noetherian local ring. Then Auslander and Bridger extended the definition to two sided noetherian rings (1969). The area really took off in the mid 90’s, with the introduction of the Gorenstein (projective and injective) modules by Enochs and Jenda (\cite{enochs:95:gorenstein}). Avramov, Buchweitz, Martsinkovsky, and Reiten proved that if the ring $R$ is both right and left noetherian and if $G$ is a finitely generated Gorenstein projective module, then Enochs' and Jenda's definition agrees with that of Auslander's and Bridger's of module of G-dimension zero.
The Gorenstein flat modules were introduced by Enochs, Jenda and Torrecillas as another extension of Auslander's Gorenstein dimension.\\

The Gorenstein homological methods have proved to be very useful in characterizing various classes of rings. Also, methods and results from Gorenstein homological algebra have successfully been used in algebraic geometry, as well as in representation theory. But the main problem in using the Gorenstein homological methods is that they can only be applied when the corresponding Gorenstein resolutions exist. So the main open problems in this area concern identifying the type of rings over which Gorenstein homological algebra works. Of course one hopes that this is the case for any ring. But so far only the existence of the Gorenstein flat resolutions was proved over arbitrary rings (in \cite{yang:14:gorflat}, 2014). The existence of the Gorenstein projective resolutions and the existence of the Gorenstein injective resolutions are still open problems. And they have been studied intensively in recent years (see for example \cite{CH}, \cite{enochs:12:gorenstein.injective.covers}, \cite{EJL}, \cite{iacob:15:gor.flat.proj}, \cite{holm:05:gor.dim}, \cite{iacob:15:gor.flat}, \cite{murfet:11:gor.proj}).



The Gorenstein (projective, injective, flat) modules are defined in terms of totally acyclic complexes.  We recall that an acyclic complex $P$ of projective $R$-modules ($R$ is an arbitrary ring) is called \emph{totally acyclic} if the complex $Hom(P,Q)$ is still exact for any projective module $Q$. A totally acyclic complex of injective modules is defined dually. And an exact complex $F$ of flat left $R$-modules is said to be \emph{$F$-totally acyclic} if $I \otimes F$ is exact for any injective right $R$-module $I$. A module $M$ is Gorenstein injective if and only if it is a cycle of a totally acyclic complex of injective modules. Dually, a module $G$ is Gorenstein projective if it is a cycle of a totally acyclic complex of projective modules. And a Gorenstein flat module is a cycle of an F-totally acyclic complex of flat modules.\\ 
An Iwanaga Gorenstein ring (\cite{iwanaga:79:gor} and \cite{iwanaga:80:gor}) is a two sided noetherian ring $R$ that has finite self injective dimension on both sides. It is known that a commutative Gorenstein ring of finite Krull dimension is Iwanaga-Gorenstein. Over an Iwanaga-Gorenstein ring the exact complexes of projective (injective, flat) modules have some very nice homological properties. More precisely, over an Iwanaga Gorenstein ring every acyclic complex of projective (injective) modules is totally acyclic. And every acyclic complexes of flat modules is F-totally acyclic over any Iwanaga-Gorenstein ring. So over such a ring the class of Gorenstein projective (injective, flat) modules coincides with that of the cycles of acyclic complexes of projectives (injective, flat modules respectively).\\

It is a natural question to consider whether or not these conditions actually characterize Gorenstein rings, or more generally whether or not  it is possible to characterize Gorenstein rings in terms of acyclic complexes of (Gorenstein) injectives, (Gorenstein) projectives and (Gorenstein) flats.  This is one of the main goals of this paper. In the commutative case we encompass and extend recent results by Murfet and Salarian in \cite{murfet:11:gor.proj} and by Iyengar and Krause in \cite{IyengarKrause}.

We give equivalent characterizations of the condition that every acyclic complex of injective (flat, projective respectively) modules is totally acyclic. It involves $\widetilde{\mathcal{A}}$ complexes, as well as dg-$\mathcal{A}$ complexes, so we recall first the following definitions.\\

\begin{definition}
Let $\mathcal{A}$ be a class of $R$-modules. An acyclic complex $X$ is in $\widetilde{\mathcal{A}}$ if $Z_j(X) \in \mathcal{A}$ for all integers j.
\end{definition}

\begin{definition}
Let $(\mathcal{A}, \mathcal{B})$ be a cotorsion pair in $R$-Mod. A complex $Y$ is a dg-$\mathcal{A}$ complex if each $Y_n \in \mathcal{A}$ and if $\mathrm{Ext}^1(Y, U)=0$ for any complex $U$ in $\widetilde{\mathcal{B}}$.

\end{definition}

Throughout the paper we use $\mathcal{GI}$ to denote the class of Gorenstein injective modules, $\mathcal{GP}$ for the class of Gorenstein projective modules, and $\mathcal{GF}$ for that of Gorenstein flat modules.\\

We prove first (Proposition 3) that, over any ring $R$, every acyclic complex of projective modules is totally acyclic if and only if the cycles of every acyclic complex of Gorenstein projective modules are Gorenstein projective. Proposition 4 shows that the analogue result for injective and Gorenstein injective modules also holds over any ring $R$. And Proposition 5 proves that when the ring $R$ is GF-closed the analogue result for flat and Gorenstein flat modules also holds.\\
We prove then (Theorem 2) that when $R$ is a left noetherian ring, a third equivalent condition can be added to those in Proposition 3. More precisely we have:\\

\textbf{Theorem 2.} Let $R$ be a left noetherian ring. The following are equivalent:\\
1. Every acyclic complex of injective left $R$-modules is totally acyclic.\\
2. Every acyclic complex of Gorenstein injective left $R$-modules is in $\widetilde{\mathcal{GI}}$.\\
3. Every complex of Gorenstein injective left $R$-modules is dg-Gorenstein injective. \\

Then, using Proposition 5, we prove the following result:\\

\textbf{Theorem 3.} Let $R$ be a left coherent ring. The following are equivalent:\\
1. Every acyclic complex of flat right $R$-modules is F-totally acyclic. \\ 
2. Every acyclic complex of Gorenstein flat right $R$-modules is in $\widetilde{\mathcal{GF}}$.\\
3. Every complex of Gorenstein flat right $R$-modules is dg-Gorenstein flat. 

We also prove (Theorem 4) that when moreover $R$ is left coherent and right $n$-perfect (that is, every flat right $R$-module has finite projective dimension $\leq n$)  then statements 1, 2, 3 in Theorem 2 are also equivalent to the following:\\
4. Every acyclic complex of projective right $R$-modules is totally acyclic.\\
5. Every acyclic complex of Gorenstein projective right $R$-modules is in $\widetilde{\mathcal{GP}}$.\\
6. Every complex of Gorenstein projective right $R$-modules is dg-Gorenstein projective.\\

We recall that a commutative noetherian ring $R$ is called Gorenstein if for each prime (resp., maximal) ideal $p$, $inj dim_{R_p} R_p < \infty$ (Bass, \cite{Bass}, Section 1). We prove:\\

\textbf{Corollary 1}  Let $R$ be a commutative noetherian ring. Then the following statements are equivalent:\\
1. $R$ is a Gorenstein ring.\\
2. Every acyclic complex of flat $R$-modules is F-totally acyclic. \\ 
3. Every acyclic complex of Gorenstein flat $R$-modules is in $\widetilde{\mathcal{GF}}$.\\
4. Every complex of Gorenstein flat $R$-modules is dg-Gorenstein flat. 

If moreover $R$ has finite Krull dimension then Theorems 2 and 4 give that the following are equivalent (Corollary 2):\\
1. $R$ is an Iwanaga-Gorenstein ring.\\
2. Every acyclic complex of injective $R$-modules is totally acyclic.\\
3. Every acyclic complex of flat $R$-modules is F-totally acyclic. \\ 
4. Every acyclic complex of Gorenstein flat $R$-modules is in $\widetilde{\mathcal{GF}}$.\\
5. Every acyclic complex of Gorenstein injective $R$-modules is in $\widetilde{\mathcal{GI}}$.\\
6. Every complex of Gorenstein injective $R$-modules is dg-Gorenstein injective. \\
7. Every complex of Gorenstein flat $R$-modules id dg-Gorenstein flat.\\
8. Every acyclic complex of projective $R$-modules is totally acyclic.\\
9. Every acyclic complex of Gorenstein projective $R$-modules is in $\widetilde{\mathcal{GP}}$.\\
10. Every complex of Gorenstein projective $R$-modules is dg-Gorenstein projective.

Our Corollary 2 improves on results by Iyengar and Krause (\cite{IyengarKrause}) and by Murfet and Salarian (\cite{murfet:11:gor.proj}). Iyengar and Krause proved that for a commutative noetherian ring $R$ with a dualizing complex, the class of acyclic complexes of injectives coincides with that of totally acyclic complexes of injectives if and only if $R$ is Gorenstein. Then Murfet and Salarian removed the dualizing complex hypothesis and characterized Gorenstein rings in terms of totally acyclic complexes of projectives. We are adding more equivalent characterizations, still under the assumption that $R$ is commutative noetherain of finite Krull dimension.

We consider then a two sided noetherian ring $R$ such that every acyclic complex of injective modules is totally acyclic. We prove (Proposition 7) that if furthermore $R$ satisfies the Auslander condition and has finitistic flat dimension then every injective $R$-module has finite flat dimension. We use this result to prove the following characterization of Iwanaga-Gorenstein rings (Theorem 7):\\
Let $R$ be a two sided noetherian ring of finitistic flat dimension that satisfies the Auslander condition. Then the following are equivalent:\\
1. $R$ is Iwanaga Gorenstein.\\
2. Every acyclic complex of injective left $R$-modules is totally acyclic and every acyclic complex of injective right $R$-modules is totally acyclic.

\section{preliminaries}

Throughout the paper $R$ will denote an associative ring with unit (non necessarily commutative). The category of left $R$-modules will be denoted by $R\textrm{-Mod}$, and the category of unbounded complexes of left $R$-modules will be denoted by $Ch(R)$. Modules are, unless otherwise explicitly stated, left modules.

We recall that an acyclic complex of projective modules $P = \ldots \rightarrow P_1 \rightarrow P_0 \rightarrow P_{-1} \rightarrow \ldots $ is said to be \emph{totally acyclic } if for any projective $R$-module $P$, the complex $\ldots \rightarrow Hom(P_{-1}, P) \rightarrow Hom(P_0, P) \rightarrow Hom(P_1, P) \rightarrow \ldots$ is still acyclic.\\
A module $G$ is Gorenstein projective if it is a cycle of such a totally acyclic complex of projective modules ($G = Z_j(P)$, for some integer $j$).\\

Dually, an acyclic complex of injective modules $I= \ldots \rightarrow I_1 \rightarrow I_0 \rightarrow I_{-1} \rightarrow \ldots $ is said to be \emph{totally acyclic } if for any injective $R$-module $I$, the complex $\ldots \rightarrow Hom(I, I_1) \rightarrow Hom(I, I_0) \rightarrow Hom(I, I_{-1}) \rightarrow \ldots$ is still acyclic.\\
A module $M$ is Gorenstein injective if it is a cycle in a totally acyclic complex of injective modules ($M =  Z_j(I)$ for some integer $j$).\\

We also recall that an acyclic complex of flat left $R$-modules, $F = \ldots \rightarrow F_1 \rightarrow F_0 \rightarrow F_{-1} \rightarrow \ldots $ is said to be \emph{F-totally acyclic } if for any injective right $R$-module $I$, the complex $\ldots \rightarrow I \otimes F_1 \rightarrow I \otimes I_0 \rightarrow I \otimes F_{-1} \rightarrow \ldots$ is still acyclic.\\
A module $H$ is Gorenstein flat if it is a cycle in an F-totally acyclic complex of flat modules ($H = Z_j(F)$ for some integer $j$).\\

A pair of classes $(\mathcal{A}, \mathcal{B})$ in an abelian category $\mathcal D$ is called a {\it cotorsion pair} if $\mathcal{B}=\mathcal{A}^{\perp}$ and $\mathcal{A}=^{\perp}\!\!\! \mathcal{B}$, where for a given class of modules $\mathcal C$, the right orthogonal class $\mathcal C^{\perp}$ is defined to be the class of objects $Y$ in $\mathcal D$ such that $\mathrm{Ext}^1(A,Y)=0$ for all $A\in \mathcal{A}$. Similarly, we define the left orthogonal class $^{\perp}{\mathcal C}$. The cotorsion pair is called {\it hereditary} if $\mathrm{Ext}^i(A,B)=0$ for all $A\in\mathcal A$, $B\in \mathcal B$, and $i\geq 1$. The cotorsion pair is called {\it complete} if for each object $M$ in $\mathcal D$ there exist short exact sequences $0\to M\to B\to A\to 0$ and $0\to B'\to A'\to M\to 0$ with $A,A'\in \mathcal A$ and $B,B'\in \mathcal B$.

Following Gillespie \cite[Definition 3.3 and Proposition 3.6]{Gill04} there are four classes of complexes in $Ch(R)$ that are associated with a cotorsion pair $(\mathcal{A}, \mathcal{B})$ in $R$-Mod:\\
1. An acyclic complex $X$ is an $\mathcal{A}$-complex if $Z_j(X) \in \mathcal{A}$ for all integers $j$. We denote by $\widetilde{\mathcal{A}}$ the class of all acyclic $\mathcal A$-complexes.\\
2. An acyclic complex $U$ is a $\mathcal{B}$-complex if $Z_j(X) \in \mathcal{B}$ for all integers $j$. We denote by $\widetilde{\mathcal{B}}$ the class of all acyclic $\mathcal B$-complexes.\\
3. A complex $Y$ is a dg-$\mathcal{A}$ complex if each $Y_n \in \mathcal{A}$ and each map $Y\to U$ is null-homotopic, for each complex $U\in \widetilde{\mathcal{B}}$. We denote by $ dg (\mathcal{A})$ the class of all dg-$\mathcal{A}$ complexes. \\
4. A complex $W$ is a dg-$\mathcal{B}$ complex if each $W_n \in \mathcal{B}$ and each map $V\to W$ is null-homotopic, for each complex $V\in \widetilde{\mathcal{A}}$. We denote by $ dg (\mathcal{B})$ the class of all dg-$\mathcal{B}$ complexes.

Yang and Liu showed in \cite[Theorem 3.5]{yang:11:cotorsion} that when $(\mathcal{A}, \mathcal{B})$ is a complete hereditary cotorsion pair in $R$-Mod, the pairs $(dg (\mathcal{A}), \widetilde{\mathcal{B}})$ and $(\widetilde{\mathcal{A}}, dg (\mathcal{B}))$ are complete (and hereditary) cotorsion pairs. Moreover, by Gillespie \cite[Theorem 3.12]{Gill04}, we have that  $\widetilde{\mathcal{A}}=dg (\mathcal{A})\bigcap \mathcal E$ and $ \widetilde{\mathcal{B}}=dg (\mathcal{B})\bigcap \mathcal E$ (where $\mathcal E$ is the class of all acyclic complexes). For example, from the (complete and hereditary) cotorsion pairs $(Proj,R\textrm{-Mod})$ and $(R\textrm{-Mod},Inj)$ one obtains the standard (complete and hereditary) cotorsion pairs $(\mathcal E,dg(Inj))$ and $(dg(Proj),\mathcal E)$. Over a left noetherian ring $R$, the pair $(^\bot \mathcal{GI}, \mathcal{GI})$ is complete hereditary cotorsion pair. This is essentially due to Krause in \cite[Theorem 7.12]{Krause} (see Enochs and Iacob \cite[Corollary 1]{enochs:12:gorenstein.injective.covers} for a precise formulation). Therefore $(dg (^\bot \mathcal {GI}), \widetilde{\mathcal {GI}})$ is a complete cotorsion pair in $Ch(R)$. \\ We recall that a ring $R$ is right $n$-perfect if each flat right $R$-module has finite projective dimension $\leq n$. Then if $R$ is left coherent and right $n$-perfect the pair $(\mathcal{GP}, \mathcal{GP}^\bot)$ is also a complete hereditary cotorsion pair in the category of right $R$-modules (see Bravo, Gillespie and Hovey \cite[Proposition 8.10]{gillespie:14:stable} or Estrada, Iacob and Odaba\c s\i\ \cite[Propostion 7]{iacob:15:gor.flat.proj}). Hence the pair $(\widetilde{\mathcal{GP}}, dg (\mathcal{GP}^\bot))$ is a complete cotorsion pair. \\
Finally, if $R$ is a left coherent ring, then $(\mathcal{GF},\mathcal{GF}^{\perp})$ is a complete hereditary cotorsion pair (where $\mathcal{GF}$ is the class of Gorenstein flat right $R$-modules). This is due to Enochs, Jenda and L\'opez-Ramos in \cite{EJL}. The class $\mathcal{GF}^{\perp}$ is known as the class of {\it Gorenstein cotorsion modules} and it is usually denoted by $\mathcal{GC}$. So the pair
$(\widetilde{\mathcal{GF}}, dg (\mathcal{GC}))$ is a complete cotorsion pair.

We also recall that given a class of modules $\mathcal{A}$, we denote by $dw(\mathcal{A})$ the class of complexes of modules, $X$, such that each component, $X_n$, is in $\mathcal{A}$.\\

\section{$\mathcal{A}$-periodic modules}

We prove in this section (Proposition 3) that over any ring $R$ the following statements are equivalent:\\
1. Every acyclic complex of projective modules is totally acyclic.\\
2. Every acyclic complex of Gorenstein projective modules is in $\widetilde{\mathcal{GP}}$.\\

Proposition 4 shows that the analogue result for injective and Gorenstein injective modules also holds over any ring $R$. And Proposition 5 proves that when the ring $R$ is GF-closed the analogue result for flat and Gorenstein flat modules also holds.\\

 We will work in a more general setting. Let $\mathcal{A}$ be a class of modules that is closed under isomorphisms.

\begin{definition}
A module $M$ is called $\mathcal{A}$-periodic module if there exists a short exact sequence $0\rightarrow M\rightarrow A\rightarrow M\rightarrow 0$ with $A\in\mathcal{A}$.

\end{definition}

Assume that $\mathcal{A}$ is closed under direct sums. Then one can see easily that the class of $\mathcal{A}$-periodic modules is closed under direct sums. Let $A = \ldots \rightarrow A_{n+1} \rightarrow A_n \rightarrow A_{n-1} \rightarrow \ldots $ be an acyclic complex. Then, by Fu and Herzog (\cite{FH}), the averaging complex of $A$ is the complex $$\oplus_{n\in \mathbb{Z}}\Sigma^nA=\cdots\to \oplus A_{n}\to \oplus A_n\to \oplus A_{n}\to\cdots$$ associated to $A$, defined as the coproduct of all the iterated suspensions and desuspensions of $A$. It is clear that the cycle of the averaging complex is a periodic $\mathcal{A}$-module, and every cycle of the complex $A$ is a direct summand of the cycle of the averaging complex. This trick was used by Christensen and Holm in \cite[Proposition 7.6]{CH} firstly, and later by Fu and Herzog in \cite{FH}, to show that a result of Neeman's \cite[Theorem 8.6 and Remark 2.15]{neeman} (every acyclic complex of projectives with flat cycles is contractible) can be deduced from a result proved by Benson and Goodearl in \cite[Theorem 2.5]{BG}. 

If we let $\mathcal{B}$ be a class of modules which is closed under isomorphisms and direct summands, then the above trick will give us the following easy observation:\\

\begin{proposition}\label{eq.periodic} Assume that $\mathcal{A}$ is closed under direct sums.
The following are equivalent:\\
1. The cycles of  every acyclic $dw \mathcal{A}$-complex belong to $\mathcal{B}$.\\
2. Every $\mathcal{A}$-periodic module belongs to  $\mathcal{B}$.
\end{proposition}

\begin{proof} 1. $\Rightarrow$ 2.  Let $M$ be an $\mathcal{A}$-periodic module. Then there is a short exact sequence $0\to M\to A\to M\to 0$ with $A\in\mathcal{A}$. One gets an acyclic complex $\cdots\to A\to A\to A\to\cdots$ immediately such that $M$ is a cycle of this complex. Then $M$ belongs to $\mathcal{B}$ by (1).\\
 2. $\Rightarrow$ 1. Let $A=\cdots\to A_{n+1}\to A_{n}\to A_{n-1}\to\cdots$ be an acyclic $dw \mathcal{A}$-complex and take the averaging complex  $\oplus_{n\in\mathbb{Z}}\Sigma^n A=\cdots\to \oplus A_{n}\to \oplus A_n\to \oplus A_{n}\to\cdots$ of $A$. Then one can see that every cycle of the averaging complex is an $\mathcal{A}$-periodic module, and hence belongs to $\mathcal{B}$. But every cycle of the complex $A$ is a direct summand of the cycle of the averaging complex, and hence belongs to $\mathcal{B}$, since $\mathcal{B}$ is closed under direct summands.
\end{proof}
If the class $\mathcal{A}$ is closed under direct products, then for a complex $A = \ldots \rightarrow A_{n+1} \rightarrow A_n \rightarrow A_{n-1} \rightarrow \ldots$ one can use the complex $$\prod_{n\in \mathbb{Z}}\Sigma^nA=\cdots\to \prod A_{n}\to \prod A_n\to \prod A_{n}\to\cdots$$ instead of the averaging complex of $A$, and prove that we still have that every cycle of $A$ is a direct summand of the cycle of the complex $\prod_{n\in \mathbb{Z}}\Sigma^nA$. So we also have
\begin{proposition}\label{eq.periodic1} Assume that $\mathcal{A}$ is closed under direct products.
The following are equivalent:\\
1. The cycles of  every acyclic $dw \mathcal{A}$-complex belong to $\mathcal{B}$.\\
2. Every $\mathcal{A}$-periodic module belongs to  $\mathcal{B}$.
\end{proposition}

Next let $\mathcal{B}$ be a class of modules closed under isomorphisms, direct summands and extensions, and assume that both $\mathcal{A}$ and $\mathcal{B}$ are closed under direct sums (resp., direct products). Moreover, we assume that $\mathcal{A} \subseteq \mathcal{B}$ and that every module in $\mathcal{B}$ appears as a cycle of some acyclic $dw \mathcal{A}$-complex. Then we have the following main theorem of this section.

\begin{theorem}The following are equivalent:\\
1. The cycles of  every acyclic $dw \mathcal{A}$-complex belong to $\mathcal{B}$.\\
2. The cycles of  every acyclic $dw \mathcal{B}$-complex belong to $\mathcal{B}$.\\
3. Every $\mathcal{A}$-periodic module belongs to $\mathcal{B}$.\\
4. Every $\mathcal{B}$-periodic module belongs to $\mathcal{B}$.\\
\end{theorem}
\begin{proof} Note that $1. \Leftrightarrow 3.$ and $2. \Leftrightarrow 4.$ follow from Proposition 1 or Proposition 2, and $4.\Rightarrow 3.$ is trivial since in this case $\mathcal{A}\subseteq\mathcal{B}$. So we only need to show $3. \Rightarrow 4 $. By our assumptions on $\mathcal{B}$ and (3), we see that $\mathcal{B}$ coincides with the class of cycles of acyclic $dw \mathcal{A}$-complexes.
Let $K$ be a $\mathcal{B}$-periodic module. Then there is an exact sequence $0\to K\to G\to K\to 0$ with $G\in\mathcal{B}$. Since $G$ belongs to $\mathcal{B}$, there is an exact sequence $0\to G_1\to E_0\to E^0\to G^1\to 0$ such that $G$ is the image of $E_0\to E^0$, $E_0$ and $E^0$ belong to $\mathcal{A}$, and $G_1$ and $G^1$ belong to $\mathcal{B}$. Composing the two epimorphisms $G\to K$ and $E_0\to G$, we get the following commutative diagram


$$\xymatrix{~&0\ar[d]&0\ar[d]&~&~\\
~&G_1\ar[d]\ar@{=}[r]&G_1\ar[d]&~&~\\
0\ar[r]&X_0\ar[r]\ar[d]&E_0\ar[r]\ar[d]&K\ar[r]\ar@{=}[d]&0\\
0\ar[r]&K\ar[r]\ar[d]&G\ar[r]\ar[d]&K\ar[r]&0\\
~&0&0&~&~}.$$

Similarly, we have a commutative diagram:\\


$$\xymatrix{~&~&0\ar[d]&0\ar[d]&~\\
0\ar[r]&K\ar[r]\ar@{=}[d]&G\ar[r]\ar[d]&K\ar[r]\ar[d]&0\\
0\ar[r]&K\ar[r]&E^0\ar[r]\ar[d]&X^0\ar[r]\ar[d]&0\\
~&~&G^1\ar[d]\ar@{=}[r]&G^1\ar[d]&~\\
~&~&0&0&~}.$$

The above two diagrams give us the following commutative diagram\\


$$\xymatrix{~&E_0\ar[rr]\ar[rd]&&E^0\ar[rd]&~\\
X_0\ar[ru]\ar[rd]&&K\ar[ru]\ar[rd]&&X^0\\
~&0\ar[ru]&&0\ar[ru]&~}$$

Now take the pushout of $K\to G$ along the exact sequence $0\to K\to X^0\to G^1\to 0$, we get the following commutative diagram\\


$$\xymatrix{~&0\ar[d]&0\ar[d]&~&~\\
0\ar[r]&K\ar[r]\ar[d]&X^0\ar[r]\ar[d]&G^1\ar[r]\ar@{=}[d]&0\\
0\ar[r]&G\ar[r]\ar[d]&Y^0\ar[r]\ar[d]&G^1\ar[r]&0\\
~&K\ar[d]\ar@{=}[r]&K\ar[d]&~&~\\
~&0&0~&&~}.$$

Since $\mathcal{B}$ is closed under extensions, and $G$ and $G^1$ belong to $\mathcal{B}$, we have $Y^0 \in\mathcal{B}$. Then there is an exact sequence $0\to Y^0\to E^1\to Y^1\to 0$ with $E^1\in\mathcal{A}$, and $Y^1$ and $Y^0$ in $\mathcal{B}$. Now composing the two monomorphisms $X^0\to Y^0$ and $Y^0\to E^1$, we get the following two commutative diagrams:


$$\xymatrix{~&~&0\ar[d]&0\ar[d]&~\\
0\ar[r]&X^0\ar[r]\ar@{=}[d]&Y^0\ar[r]\ar[d]&K\ar[r]\ar[d]&0\\
0\ar[r]&X^0\ar[r]&E^1\ar[r]\ar[d]&X^1\ar[r]\ar[d]&0\\
~&~&Y^1\ar[d]\ar@{=}[r]&Y^1\ar[d]&~\\
~&~&0&0&~}$$

and

$$\xymatrix{~&E_0\ar[rr]\ar[rd]&&E^0\ar[rd]\ar[rr]&&E^1\ar[rd]\\
X_0\ar[ru]\ar[rd]&&K\ar[ru]\ar[rd]&&X^0\ar[ru]\ar[rd]&&X^1\\
~&0\ar[ru]&&0\ar[ru]&&0\ar[ru]&~}.$$ Continuing this procedure, we get the follow exact sequence
$$E_0\to E^0\to E^1\to\cdots .$$ Using the sequence $0\to G_1\to X_0\to K\to 0$ and a dual argument to the one above, finally we get an acyclic $\mathcal{A}$-complex
 $$\cdots\to E_1\to E_0\to E^0\to E^1\to\cdots$$ with $K$ a cycle of this complex, so $K \in \mathcal{B}$, as desired.
\end{proof}

As applications, we obtain the results mentioned in the beginning of this section.\\

\begin{proposition}The following are equivalent:\\
1. Every acyclic complex of projective modules is totally acylic;\\
2. The cycles of every acyclic complex of Gorenstein projective modules are Gorenstein projective.
\end{proposition}
\begin{proof}
Let $\mathcal{A}$ be the class of projective modules, and let $\mathcal{B}$ be the class of Gorenstein projective modules. Then $\mathcal{A}$ and $\mathcal{B}$ are closed under direct sums, direct summands, and extensions. Note that $2 \Rightarrow 1$ is trivial. For $1 \Rightarrow 2$, it is clear that if every acyclic complex of projective modules is totally acylic, then every cycle of an acyclic complex of projective modules is Gorenstein projective. Therefore by Theorem 1, we get that the cycles of every acyclic complex of Gorenstein projective modules are Gorenstein projective.
\end{proof}

If $\mathcal{A}$ is the class of injective modules, and $\mathcal{B}$ is the class of Gorenstein injective modules, then by a similar argument as above, we have the following result.

\begin{proposition}The following are equivalent:\\
1. Every acyclic complex of injective modules is totally acylic.\\
2. The cycles of every acyclic complex of Gorenstein injective modules are Gorenstein injective.
\end{proposition}

Recall that a ring $R$ is GF-closed if the class Gorenstein flat modules is closed under extensions. \\
If $\mathcal{A}$ is the class of flat modules, and $\mathcal{B}$ is the class of Gorenstein flat modules, then over a GF-closed ring $R$ we have the following result.
\begin{proposition}Let $R$ be a GF-closed ring. The following are equivalent:\\
1. Every acyclic complex of flat modules is totally F-acylic.\\
2. The cycles of every acyclic complex of Gorenstein  flat modules are Goresntein flat.
\end{proposition}

\section{Totally acyclic complexes of injective, projective, flat modules}



Theorem 2 below shows that when $R$ is a left noetherian ring, we can add a third equivalent statement to those in Proposition 4. The proof uses the fact that over such a ring $(^\bot \mathcal{GI}, \mathcal{GI})$ is a complete hereditary cotorsion pair, and therefore we have that the pair $(dg (^\bot \mathcal{GI}), \widetilde{\mathcal{GI}})$ is a complete cotorsion pair in $Ch(R)$.

\begin{theorem}
Let $R$ be a left noetherian ring. The following are equivalent:\\
1. Every acyclic complex of injective $R$-modules is totally acyclic.\\
2. Every acyclic complex of Gorenstein injective $R$-modules is in $\widetilde{\mathcal{GI}}$.\\
3. Every complex of Gorenstein injective $R$-modules is a dg-Gorenstein injective complex. \\
\end{theorem}

\begin{proof}
1. $\Leftrightarrow$ 2. By Proposition 4.\\
2. $\Rightarrow$ 3. Let $X$ be a complex of Gorenstein injective $R$-modules. Since $(\mathcal{E}, dg (Inj))$ is a complete cotorsion pair, there is an exact sequence $0 \rightarrow A \rightarrow B \rightarrow X \rightarrow 0$ with $A$ a DG-injective complex and with $B$ an acyclic complex. Then for each $n$ there is an exact sequence $0 \rightarrow A_n \rightarrow B_n \rightarrow X_n \rightarrow 0$ with $A_n$ injective and with $X_n$ Gorenstein injective. It follows that each $B_n$ is Gorenstein injective. So $B$ is an acyclic complex of Gorenstein injective modules; by (2), $B$ is in $\widetilde{\mathcal{GI}}$, and therefore in $dg (\mathcal{\mathcal{GI}})$. \\
Let $Y \in \widetilde{^\bot \mathcal{GI}}$. The exact sequence $0 \rightarrow A \rightarrow B \rightarrow X \rightarrow 0$ gives an exact sequence $0 =\mathrm{ Ext}^1(Y, B) \rightarrow \mathrm{Ext}^1(Y,X) \rightarrow \mathrm{Ext}^2(Y, A)=0$ (since $Y$ is acyclic and $A$ is a DG-injective complex). It follows that $\mathrm{Ext}^1(Y,X)=0$ for any $Y \in \widetilde{^\bot \mathcal{GI}}$, so $X \in dg (\mathcal{GI})$. So we have that $dw (\mathcal{GI}) \subseteq dg (\mathcal{GI})$. The other inclusion always holds, thus $dg (\mathcal{GI}) = dw (\mathcal{GI})$.\\
3. $\Rightarrow$ 1. Let $X$ be an acyclic complex of injective $R$-modules. In particular, $X \in dw (\mathcal{GI})$ and by (3), $X$ is in $dg (\mathcal{GI})$. Since $X \in dg (\mathcal{GI})$ and $X$ is acyclic, it follows that $X \in \widetilde{\mathcal{GI}}$, and therefore $Z_n(X) \in \mathcal{GI}$ for all $n$. Thus $X$ is a totally acyclic complex.
\end{proof}


We recall that any left coherent ring is right GF-closed (see Bennis \cite[Proposition 2.2(1)]{B}).\\ The dual result of Theorem 2 (for flat/Gorenstein flat modules) is the following:\\

\begin{theorem}
Let $R$ be a left coherent ring. Then the following are equivalent.\\
1. Every acyclic complex of flat right $R$-modules is F-totally acyclic. \\ 
2. Every acyclic complex of Gorenstein flat right $R$-modules is in $\widetilde{\mathcal{GF}}$.\\
3.  Every complex of Gorenstein flat right $R$-modules is a dg-Gorenstein flat complex. 
\end{theorem}

\begin{proof}
1. $\Leftrightarrow$ 2. by Proposition 5.\\
2. $\Rightarrow$ 3. Let $X$ be a complex of Gorenstein flat right $R$-modules. Since $(dg(Proj), \mathcal{E})$ is a complete cotorsion pair, there exists an acyclic sequence $0 \rightarrow X \rightarrow C \rightarrow D \rightarrow 0$ with $D \in dg(Proj)$ and with $C$ an acyclic complex. Then for each $n$ we have an exact sequence $0 \rightarrow X_n \rightarrow C_n \rightarrow D_n \rightarrow 0$ with both $D_n$ and $X_n$ Gorenstein flat right $R$-modules. It follows that each $C_n$ is Gorenstein flat right $R$-module. Thus $C$ is an acyclic complex of Gorenstein flat right $R$-modules, so by (2), $C$ is in $\widetilde{\mathcal{GF}}$.\\
 Let $A \in \widetilde{\mathcal{GC}}$. The exact sequence $0 \rightarrow X \rightarrow C \rightarrow D \rightarrow 0$ gives an exact sequence $0 = \mathrm{Ext}^1 (C,A) \rightarrow \mathrm{Ext}^1(X,A) \rightarrow\mathrm{Ext}^2(D,A) = 0$ (since $A$ is acyclic and $D$ is DG-projective). It follows that $\mathrm{Ext}^1(X,A) = 0$ for any $A \in \widetilde{\mathcal{GC}}$, so $X \in dg (\mathcal{GF})$.\\
3. $\Rightarrow$ 1. Let $Y$ be an acyclic complex of flat right $R$-modules. By (3), $Y \in dg (\mathcal{GF})$. Since $Y$ is also acyclic it follows that $Y$ is in $\widetilde{\mathcal{GF}}$. Therefore $Z_n(Y)$ is Gorenstein flat right $R$-module for each $n$. So $Y$ is F-totally acyclic.
\end{proof}

Using Theorem 3 we obtain the following characterization of commutative Gorenstein rings.\\

\begin{corollary}
Let $R$ be a commutative noetherian ring. The following are equivalent:

1. $R$ is Gorenstein.\\
2. Every acyclic complex of flat $R$-modules is F-totally acyclic.\\
3. Every acyclic complex of Gorenstein flat $R$-modules is in $\widetilde{\mathcal{GF}}$.\\
4. Every complex of Gorenstein flat $R$-modules is  dg-Gorenstein flat.
\end{corollary}

\begin{proof}
1. $\Leftrightarrow$ 2. by Murfet and Salarian \cite[Theorem 4.27]{murfet:11:gor.proj}.\par\noindent
By Theorem 3, (2), (3) and (4) are equivalent. 
\end{proof}



We show that if $R$ is a left coherent and right $n$-perfect ring, then the equivalent characterizations from Theorem 3 can be extended to include the analogue results for the Gorenstein projective modules. The proof uses the fact that over such a ring $R$ the pair $(\mathcal{GP}, \mathcal{GP}^\bot)$ is a complete hereditary pair. As noted in Section 2, this gives a complete cotorsion pair, $(dg(\mathcal{GP}), \widetilde{\mathcal{GP}^\bot})$, in the category of complexes of right $R$-modules. 

\begin{theorem}
Let $R$ be a left coherent and right $n$-perfect ring. The following statements are equivalent:\\
1. Every acyclic complex of flat right $R$-modules is F-totally acyclic. \\ 
2. Every acyclic complex of Gorenstein flat right $R$-modules is in $\widetilde{\mathcal{GF}}$.\\
3. Every complex of Gorenstein flat right $R$-modules is a dg-Gorenstein flat complex.\\ 
4. Every acyclic complex of projective right $R$-modules is totally acyclic.\\
5. Every acyclic complex of Gorenstein projective  right $R$-modules is in $\widetilde{\mathcal{GP}}$.\\
6. Every complex of Gorenstein projective right $R$-modules is a dg-Gorenstein projective complex.
\end{theorem}

\begin{proof}
By Theorem 3, statements (1), (2) and (3) are equivalent. And by Proposition 3, statements (4) and (5) are equivalent.\\
2 $\Rightarrow$ 5. Let $X$ be an acyclic complex of Gorenstein projective right $R$-modules. Since the ring is coherent and right $n$-perfect, by Christensen, Frankild and Holm \cite[Proposition 3.7]{christensen:06:ongorenstein}, every Gorenstein projective right $R$-module is Gorenstein flat. So $X \in \widetilde{\mathcal{GF}}$. Then, for each $j$, $Z_j(X) \in \mathcal{GF}$. By \cite[Proposition 5]{iacob:15:gor.flat.proj}, $G.p.d. Z_j(X) \le n$. Since we have an exact sequence $0 \rightarrow Z_{n+j}(X) \rightarrow X_{n+j-1} \rightarrow X_{n+j-2} \rightarrow \ldots \rightarrow X_{j+1} \rightarrow Z_j(X) \rightarrow 0$ with $Z_j(X)$ Gorenstein flat and all the $X_i$'s Gorenstein projective right $R$-modules, it follows that $Z_{j+n}(X)$ is Gorenstein projective for all integers $j$. But then by replacing $j$ with $j-n$ we obtain that $Z_j(X)$ is Gorenstein projective for all $j$.\\
5. $\Rightarrow$ 6. Let $X$ be a complex of Gorenstein projective right $R$-modules. There exists an exact sequence $0 \rightarrow X \rightarrow C \rightarrow D \rightarrow 0$ with $C$ acyclic and with $D$ a DG-projective complex. For each $j$, the exact sequence $0 \rightarrow X_n \rightarrow C_n \rightarrow D_n \rightarrow 0$ with both $X_n$ and $D_n$ Gorenstein projective right modules gives that each $C_n$ is Gorenstein projective. Thus $C$ is an acyclic complex of Gorenstein projective right $R$-modules, so by (5), $C$ is in $\widetilde{\mathcal{GP}}$. \\ Let $A \in \widetilde{\mathcal{GP}^\bot}$. The exact sequence $0 \rightarrow X \rightarrow C \rightarrow D \rightarrow 0$ gives an exact sequence $0 = \mathrm{Ext}^1(C,A) \rightarrow \mathrm{Ext}^1(X,A) \rightarrow \mathrm{Ext}^2(D,A) =0$ (since $D$ is DG-projective, and $A$ is acyclic).\\
Since $\mathrm{Ext}^1(X,A)=0$ for any $A \in \widetilde{\mathcal{GP}^\bot}$, and $(dg(\mathcal{GP}), \widetilde{\mathcal{GP}^\bot})$ is a cotorsion pair it follows that $X$ is a dg-Gorenstein projective complex.\\
6 $\Rightarrow$ 2. Let $X$ be an acyclic complex of Gorenstein flat right $R$-modules. Consider a partial projective resolution of $X$:\\
$0 \rightarrow Y \rightarrow P_{n-1} \rightarrow \ldots \rightarrow P_0 \rightarrow X \rightarrow 0$. Since $R$ is right $n$-perfect and each $X_j$ is Gorenstein flat by \cite[Proposition 5]{iacob:15:gor.flat.proj} we have that for each $j$, $G.p.d. X_j \le n$, so each $Y_j$ is Gorenstein projective. Then $Y$ is an acyclic complex of Gorenstein projective right $R$-modules, so, by (6), $Y$ is in $\widetilde{\mathcal{GP}}$. Therefore $Z_j(Y) \in \mathcal{GP}$ for all $j$, so the exact sequence $0 \rightarrow Z_j(Y) \rightarrow Z_j(P_{n-1}) \rightarrow \ldots \rightarrow Z_j(P_0) \rightarrow Z_j(X) \rightarrow 0$ gives that $G.p.d. Z_j(X) \le n$ for all $j$. By \cite[Proposition 3.7]{christensen:06:ongorenstein} we have $G.f.d. Z_j(X) \le G.p.d. Z_j(X) \le n$. The exact sequence $0 \rightarrow Z_{j+n}(X) \rightarrow {X}_{j+n-1} \rightarrow \ldots \rightarrow {X}_{j+1} \rightarrow Z_j(X) \rightarrow 0$ with all $X_i$ Gorenstein flat and with $G.f.d. Z_j(X) \le n$ gives that $Z_{j+n}(X) \in \mathcal{GF}$ for all $j$. Then by replacing $j$ with $j-n$ we obtain that $Z_j(X)$ is Gorenstein flat for all integers $j$. So $X \in \widetilde{\mathcal{GF}}$.\\
\end{proof}

Using Theorem 2 and Theorem 4 we obtain the following:\\

\begin{corollary}
Let $R$ be a commutative noetherian ring of finite Krull dimension (for instance if $R$ has a dualizing complex). The following are equivalent:\\
1. $R$ is an Iwanaga-Gorenstein ring.\\
2. Every acyclic complex of injective modules is totally acyclic.\\
3. Every acyclic complex of flat $R$-modules is F-totally acyclic. \\ 
4. Every acyclic complex of Gorenstein flat $R$-modules is in $\widetilde{\mathcal{GF}}$.\\
5. Every acyclic complex of Gorenstein injective $R$-modules is in $\widetilde{\mathcal{GI}}$.\\
6. Every complex of Gorenstein injective $R$-modules is dg-Gorenstein injective. \\
7. Every complex of Gorenstein flat $R$-modules is dg-Gorenstein flat.\\
8. Every acyclic complex of projective $R$-modules is totally acyclic.\\
9. Every acyclic complex of Gorenstein projective $R$-modules is in $\widetilde{\mathcal{GP}}$.\\
10. Every complex of Gorenstein projective $R$-modules is dg-Gorenstein projective.

\end{corollary}

\begin{proof}
1 $\Rightarrow$ 2 follows from \cite[Theorem 10.1.13(1)]{enochs:00:relative}.\par\noindent\\
By Theorem 2, 2 $\Leftrightarrow$ 5 $\Leftrightarrow$ 6.\\
2. $\Rightarrow$ 3. Let $F$ be an acyclic complex of flat modules. Then $F^+$ is an acyclic complex of injective modules. By hypothesis $F^+$ is totally acyclic.
This means that $Hom(I,F^+)$ is acyclic for every injective module $I$. But $Hom(I,F^+)\simeq (I\otimes F)^+$. So $(I\otimes F)^+$ is acyclic, which implies that $I\otimes F$ is acyclic, for every injective $I$. That is, $F$ is F-totally acyclic.\par\noindent
By Theorem 4, we have that 3 $\Leftrightarrow$ 4 $\Leftrightarrow$ 7 $\Leftrightarrow$ 8 $\Leftrightarrow$ 9 $\Leftrightarrow$ 10.\\
3 $\Leftrightarrow$ 1. By Murfet and Salarian \cite[Theorem 4.27]{murfet:11:gor.proj}, the ring $R$ is Gorenstein. Since $R$ has finite Krull dimension, it follows that $inj. dim_R R < \infty$ (see for example \cite{Bass}, Section 1). So $R$ is an Iwanaga-Gorenstein ring.
\end{proof}

One of the main open problems in Gorenstein homological algebra is: ``What is the most general type of ring over which the class of Gorenstein injective modules is (pre)covering (preenveloping respectively)?". We give a sufficient condition in order for $\mathcal{GI}$ be both covering and enveloping.\\
We will use the folowing.\\

\begin{proposition}
Let $R$ be a two sided noetherian ring such that every acyclic complex of injective $R$-modules is totally acyclic. Then the character module of any Gorenstein injective $R$-module is a Gorenstein flat right $R$-module.
\end{proposition}

\begin{proof}
Let $_RG$ be a Gorenstein injective module. Then there exists an acyclic complex of injective $R$-modules $I = \ldots \rightarrow I_1 \rightarrow I_0 \rightarrow I_{-1} \rightarrow \ldots $ with $G = Z_0(I)$. Then $I^{++}$ is an acyclic complex of injective $R$-modules, so by hypothesis, $I^{++}$ is totally acyclic. Therefore $G^{++} = Z_0(I^{++})$ is Gorenstein injective. Since $(G^+)^+$ is Gorenstein injective, it follows that $G^+$ is Gorenstein flat (by Holm \cite{holm:05:gor.dim}, Theorem 3.6).
\end{proof}

\begin{theorem}
Let $R$ be a two sided noetherian ring such that every acyclic complex of injective $R$-modules is totally acyclic. Then the class of Gorenstein injective modules is both covering and enveloping in $R\textrm{-Mod}$.
\end{theorem}

\begin{proof}
By Proposition 6, over such a ring $R$ the character module of any Gorenstein injective $R$-module is Gorenstein flat. By Iacob \cite[Theorems 3 and 5]{iacob:15:gor.inj}, the class of Gorenstein injective $R$-modules is both covering and enveloping.

\end{proof}


\begin{theorem}
Let $R$ be a two sided noetherain ring such that every acyclic complex of injective $R$-modules is totally acyclic. Then the class of Gorenstein flat right $R$-modules is preenveloping in $R$-Mod.
\end{theorem}

\begin{proof}
Since over such a ring the character modules of Gorenstein injective modules are Gorenstein flat the result follows from Iacob \cite[Theorem 1]{iacob:15:gor.flat}.
\end{proof}

\section{rings that satisfy the Auslander condition}

We recall that Bass proved that a commutative noetherian ring $R$ is an Iwanaga-Gorenstein ring if and only if the flat dimension of the $i$th term in a minimal injective resolution of $R$ is at most $i-1$, for all $i \ge 1$. In the non-commutative case, Auslander proved that this condition is left-right symmetric (see Fossum, Griffith and Reiten \cite[Theorem 3.7]{fossum:75:auslander}). In this case the ring is said to satisfy the \emph{Auslander condition}.\\
In \cite{huang:14:auslander} Huang introduces the notion of modules satisfying the Auslander condition.
We recall the definition (\cite{huang:14:auslander}): given a left noetherian ring $R$, a left $R$-module $M$ is said to satisfy the Auslander
condition if the flat dimension of the $i$th term in the minimal injective
resolution of $M$ is at most $i-1$ for any $i \ge 1$.\\
We also recall the following:\\
\textbf{Theorem} (this is part of \cite[Theorem 1.3]{huang:14:auslander}) If $R$ is a left noetherian ring then the following are
equivalent:\\
1. $_RR$ satisfies the Auslander condition.\\
2. $fd_R E^0(M) \le fd_RM$ for any $_RM$, where $E^0(M)$ is the injective envelope of $M$.\\
If moreover $R$ is left and right noetherian then the statements above are also equivalent to:\\
3. the opposite version of (i) ($1 \le i \le 2$).\\

 We recall that a ring $R$ has {\it finite finitistic flat dimension} if the maximum of flat dimensions among the modules with finite flat dimension is finite. In the following we prove (Proposition 6) that if $R$ is two sided noetherian of finite finitistic flat dimension, such that $R$ satisfies the Auslander condition and every acyclic complex of injective $R$-modules is totally acyclic then every injective $R$-module has finite flat dimension.\\
We recall that a module $M$ is \emph{strongly cotorsion} if $\mathrm{Ext}^1(F,M)=0$ for any module $F$ of finite flat dimension. By Yan \cite[Theorem 2.5 and Proposition 2.14]{yan:10:cotorsion}, if $R$ has finite finitistic flat dimension then $(\mathcal{F}, \mathcal{SC})$ is a complete hereditary cotorsion pair (where $\mathcal{F}$ denotes the class of modules of finite flat dimension and $\mathcal{SC}$ is the class of strongly cotorsion modules). We use this result to prove (Theorem 7) that if every acyclic complex of injective left $R$-modules is totally acyclic and every acyclic complex of injective right $R$-modules is totally acyclic then $R$ is an Iwanaga-Gorenstein ring.\\


We start with the following result:\\

\begin{lemma}
If $R$ is two sided noetherian such that every acyclic complex of injective modules is totally acyclic, then any Gorenstein injective module is strongly cotorsion.
\end{lemma}

\begin{proof}
By Proposition 6 the character module of any Gorenstein injective $R$-module is Gorenstein flat right $R$-module. By Iacob \cite[Lemma 2]{iacob:15:gor.inj}, we have that $K \in ^\bot\!\! \mathcal{GI}$ if and only if $K^+$ is Gorenstein cotorsion right $R$-module. Since for any flat module $K$, we have that $K^+$ is an  injective right $R$-module, it follows that any flat module is in $^\bot \mathcal{GI}$.\\
Let $C$ be a module of finite flat dimension. Then there is an exact sequence $0 \rightarrow F_n \rightarrow \ldots \rightarrow F_0 \rightarrow C \rightarrow 0$ with each $F_j$ flat. Let $G$ be a Gorenstein injective $R$-module. Then by the above $\mathrm{Ext}^l(C,G)=0$ for all $l \ge n+1$.\\
Also, there is an exact sequence $0 \rightarrow G_n \rightarrow E_{n-1} \rightarrow \ldots \rightarrow E_0 \rightarrow G \rightarrow 0$ with each $E_j$ injective and with all $\mathrm{Ker}(E_j \rightarrow E_{j-1})$ Gorenstein injective. Then $\mathrm{Ext} ^1 (C,G) \simeq \mathrm{Ext}^{n+1}(C, G_n) = 0$. So $G$ is strongly cotorsion.
\end{proof}

\begin{lemma}
Let $R$ be a two sided noetherian ring that satisfies the Auslander condition and such that every acyclic complex of injective $R$-modules is totally acyclic. Then every strongly cotorsion module has Gorenstein injective dimension $\le 1$.
\end{lemma}

\begin{proof}
Let $M$ be a strongly cotorsion module. Consider the exact sequence $0 \rightarrow M \rightarrow A \rightarrow L \rightarrow 0$ with $A$ injective. Since both $A$ and $M$ are strongly cotorsion it follows that $L$ is also strongly cotorsion. Since the injective envelope of $_R R$ is flat it follows (from Enochs and Huang \cite[Theorem 4.4(5)]{enochs:12:huang}) that the injective cover $I_0\to L$ is surjective. By Wakamatsu's lemma (\cite[Corollary 7.2.3]{enochs:00:relative}) $J_0=\mathrm{Ker}(I_0\to L)\in Inj^\bot$. Hence we have the short exact sequence $0 \rightarrow J_0 \rightarrow I_0 \rightarrow L \rightarrow 0$ with $I_0$ injective and $J_0 \in Inj^\bot$. Since $A$ is injective and $I_0 \rightarrow L$ is an injective precover, there is a commutative diagram:\\

\[
\begin{diagram}
\node{0}\arrow{e}\node{M}\arrow{s,r}{u}\arrow{e}\node{A}\arrow{s,r}{u}\arrow{e,t}{f}\node{L}\arrow{s,=}\arrow{e}\node{0}\\
\node{0}\arrow{e}\node{J_0}\arrow{e}\node{I_0}\arrow{e,t}{g}\node{L}\arrow{e}\node{0}
\end{diagram}
\]

So we have an exact sequence: $0 \rightarrow M \rightarrow J_0 \oplus A \rightarrow I_0 \rightarrow 0$ with both $M$ and $I_0$ strongly cotorsion modules. It follows that $J_0$ is strongly cotorsion. Then, by the same reasoning as above, there exists an exact sequence $0 \rightarrow J_1 \rightarrow I_1 \rightarrow J_0 \rightarrow 0$ with $I_1$ an injective module and with $J_1$ in $Inj^\bot$. In fact, since $J_0 \in Inj^\bot$, we have that $\mathrm{Ext}^1 (E, J_0) = \mathrm{Ext}^2(E, J_1)=0$ for any injective $R$-module $E$. \\
We show that $J_1$ is a strongly cotorsion module.\\
Let $F$ be a module of finite flat dimension. Consider the exact sequence $0 \rightarrow F \rightarrow E \rightarrow D \rightarrow 0$ with $E$ the injective envelope of $F$. Since $R$ satisfies the Auslander condition, $E$ has finite flat dimension. It follows that $D$ is also a module of finite flat dimension, so $\mathrm{Ext}^1(D, J_0)=0$. The exact sequence  $0 \rightarrow F \rightarrow E \rightarrow D \rightarrow 0$ gives a long exact sequence $$0 = \mathrm{Ext}^1(E,J_1) \rightarrow {\rm Ext}^1(F,J_1) \rightarrow \mathrm{Ext}^2(D,J_1) \rightarrow \mathrm{Ext}^2(E, J_1)=0.$$ So $\mathrm{Ext}^1(F,J_1) \simeq \mathrm{Ext}^2(D,J_1)$.\\
Also, the exact sequence $0 \rightarrow J_1 \rightarrow I_1 \rightarrow J_0 \rightarrow 0$ gives the exact sequence: $0 =\mathrm{Ext}^1(D, J_0) \rightarrow \mathrm{Ext}^2(D, J_1) \rightarrow \mathrm{Ext}^2(D,I_1) = 0$. Thus $\mathrm{Ext}^2(D, J_1)=0$, and by the above, $\mathrm{Ext}^1(F,J_1)=0$ for any $R$-module $F$ of finite flat dimension. So $J_1$ is a strongly cotorsion module, and therefore its injective cover is a surjective map. Continuing this process, we obtain an acyclic left injective resolution of $L$= $\ldots \rightarrow I_2 \rightarrow I_1 \rightarrow I_0 \rightarrow L \rightarrow 0$. Pasting it together with a right injective resolution of $L$, we obtain an acyclic complex of injective modules: $\ldots \rightarrow I_2 \rightarrow I_1 \rightarrow I_0 \rightarrow E^0 \rightarrow E^1 \rightarrow \ldots $. By hypothesis, this is a totally acyclic complex. So $L$ is Gorenstein injective. Then the exact sequence $0 \rightarrow M \rightarrow A \rightarrow L \rightarrow 0$ with both $A$ and $L$ Gorenstein injective modules gives that $G.i.d. M \le 1$.\\

\end{proof}

\begin{lemma}
Let $R$ be a two sided noetherian ring that satisfies the Auslander condition. If $V$ is a strongly cotorsion module of finite flat dimension then $V$ is injective.
\end{lemma}

\begin{proof}
Consider the exact sequence $0 \rightarrow V \rightarrow \mathcal{E}(V) \rightarrow W \rightarrow 0$ with $\mathcal{E}(V)$ the injective envelope of $V$. Since $R$ satisfies the Auslander condition, $f.d. (\mathcal{E}(V)) < \infty$. It follows that $W$ also has finite flat dimension. Since $V$ is strongly cotorsion, $\mathrm{Ext}^1 (W,V)=0$. So the sequence is split acyclic, and therefore $\mathcal{E}(V) \simeq V \oplus W$. Thus $V$ is an injective module.
\end{proof}

\begin{proposition}
Let $R$ be a two sided noetherian ring of finite finitistic flat dimension and that satisfies the Auslander condition. If moreover every acyclic complex of injective $R$-modules is totally acyclic then every strongly cotorsion $R$-module is Gorenstein injective.
\end{proposition}

\begin{proof}
Let $M$ be a strongly cotorsion $R$-module. Then its flat cover is injective (by \cite[Theorem 4.4(5)]{enochs:12:huang}), and therefore there exists an exact sequence $0 \rightarrow J \rightarrow I \rightarrow M \rightarrow 0$ with $I$ injective and with $J \in Inj^\bot$. Since $(\mathcal{F}, \mathcal{SC})$ is a complete cotorsion pair there is also an exact sequence $0 \rightarrow J \rightarrow U \rightarrow V \rightarrow 0$ with $U$ strongly cotorsion and with $V$ of finite flat dimension. \\
Since $I$ is an injective module we have a commutative diagram\\

\[
\begin{diagram}
\node{0}\arrow{e}\node{J}\arrow{s,=}\arrow{e} \node{U}\arrow{s} \arrow{e} \node{V}\arrow{s}\arrow{e}\node{0}\\
\node{0}\arrow{e}\node{J}\arrow{e}\node{I}\arrow{e}\node{M}\arrow{e}\node{0}
\end{diagram}
\]

and therefore an exact sequence $0 \rightarrow U \rightarrow I \oplus V \rightarrow M \rightarrow 0$.\\
Both $M$ and $U$ are strongly cotorsion, so $V$ is also strongly cotorsion. But $V$ has finite flat dimension. So by Lemma 3, $V$ is injective.\\
And by Lemma 2, $G.i.d. U \le 1$. The exact sequence $0 \rightarrow U \rightarrow I \oplus V \rightarrow M \rightarrow 0$ with $I \oplus V$ injective and with $G.i.d. U \le 1$ gives that $M$ is Gorenstein injective.

\end{proof}

\begin{corollary}
Let $R$ be a two sided noetherian ring of finitistic flat dimension and satisfies the Auslander condition. If moreover every acyclic complex of injective $R$-modules is totally acyclic then the class of strongly cotorsion $R$-modules coincides with that of the Gorenstein injective modules.
\end{corollary}

\begin{proof}
This follows from Lemma 1  and  Proposition 7.
\end{proof}

We can prove now:\\

\begin{proposition}
Let $R$ be a two sided noetherian ring  of finitistic flat dimension such that $R$ satisfies the Auslander condition and every acyclic complex of injective $R$-modules is totally acyclic. Then every injective $R$-module has finite flat dimension.
\end{proposition}

\begin{proof}
By Corollary 3 above we have that $\mathcal{GI} = \mathcal{SC}$ in this case. It follows that $^\bot \mathcal{GI} = \mathcal{F}$ with $\mathcal{F}$ the class of modules of finite flat dimension. Since the class of injective modules is contained in $^\bot \mathcal{GI}$, we have that $Inj \subseteq \mathcal{F}$.
\end{proof}

We can give now the following characterization of noncommutative Iwanaga-Gorenstein rings:\\

\begin{theorem}
Let $R$ be a two sided noetherian ring  of finitistic flat dimension that satisfies the Auslander condition. The following are equivalent:\\
1. $R$ is an Iwanaga-Gorenstein ring.\\
2. Every acyclic complex of injective left $R$-modules is totally acyclic and every acyclic complex of injective right $R$-modules is totally acyclic.
\end{theorem}

\begin{proof}
1 $\Rightarrow$ 2. is known (\cite[Theorem 10.1.13]{enochs:00:relative}).\\
2 $\Rightarrow$ 1. Let $F$ be a flat left $R$-module. Then $F^+$ is an injective right $R$-module. By Proposition 7, $flat. dim. F^+ < \infty$. Therefore $inj.dim. F^{++} < \infty$. Since $R$ is left noetherian and $F \subseteq F^{++}$ is a pure submodule it follows that $inj. dim F \le inj. dim. F^{++} < \infty$ (\cite[Lemma 9.1.5]{enochs:00:relative}). In particular, for $F = R$ we obtain that $inj.dim._R R < \infty$.\\
Since $R$ is left and right noetherian and $inj.dim._R R < \infty$ it follows (\cite{enochs:00:relative}, Proposition 9.1.6) that $inj.dim R_R < \infty$. Thus $R$ is an Iwanaga-Gorenstein ring.
\end{proof}

\bibliographystyle{plain}



\bibliographystyle{plain}

\begin{thebibliography}{1}

\bibitem{Bass}
H. Bass.
\newblock On the ubiquity of Gorenstein rings.
\newblock {\em Math. Zeit.},82: 8--28, 1963.



\bibitem{B}
D. Bennis.
\newblock Rings over which the class of Gorenstein flat modules is closed under extensions.
\newblock {\em Comm. Algebra}, 37(3):855--868, 2009.


\bibitem{BG}
D. Benson and K. Goodearl.
\newblock Periodic flat modules, and flat modules for finite groups.
\newblock {\em Pac. J. Math.}, 196(1):45--67, 2000.

\bibitem{gillespie:14:stable}
D. Bravo, J. Gillespie and M. Hovey.
\newblock The stable module category of a general ring.
\newblock {preprint}, arxiv:1210.0196.

\bibitem{CH}
L.W. Christensen and H.~Holm.
\newblock The direct limit closure of perfect complexes.
\newblock {\em J. Pure App. Algebra}, 219:449--463, 2015.

\bibitem{christensen:06:ongorenstein}
L.W. Christensen, A.~Frankild and H.~Holm.
\newblock {On Gorenstein projective, injective and flat dimensions ~ -- A
  functorial description with applications}.
\newblock {\em J. of Algebra}, 302:231--279, 2006.


\bibitem{enochs:14:houston}
E.E. Enochs, S. Estrada and A. Iacob.
\newblock {Cotorsion pairs, model structures and homotopy categories}.
\newblock {\em {Houston J. Math.}} 40(1):43--61, 2014.



\bibitem{enochs:12:huang}
E.E. Enochs and Z. Huang
\newblock {Injective Envelopes and (Gorenstein) Flat Covers}.
\newblock {\em {Alg. Rep. Theory}}, 15:1131--1145, 2012.


\bibitem{enochs:12:gorenstein.injective.covers}
E.E. Enochs and A. Iacob.
\newblock {Gorenstein injective covers and envelopes over noetherian rings}.
\newblock {\em Proc. Amer. Math. Soc.}, 143:5--12, 2015.


\bibitem{enochs:95:gorenstein}
E.E. Enochs and O.M.G. Jenda.
\newblock {Gorenstein injective and projective modules}.
\newblock {\em {Math. Zeit.}}, 220:611--633, 1995.



\bibitem{enochs:00:relative}
E.E. Enochs and O.M.G. Jenda.
\newblock {\em Relative Homological Algebra}.
\newblock Walter de Gruyter, 2000.
\newblock De Gruyter Exposition in Math; 30.

\bibitem{EJL}
E.E. Enochs, O.M.G. Jenda and J.A. L\'{o}pez-Ramos.
\newblock The existence of Gorenstein flat covers.
\newblock {\em Math. Scand.}, 94:46--62, 2004.


\bibitem{enochs:94:gorflat}
E.E. Enochs, O.M.G. Jenda and B. Torrecillas.
\newblock Gorenstein flat modules.
\newblock {\em J. Nanjing Univ.}, 10:1--9, 1994.

\bibitem{FH} X.Fu, I. Herzog, Periodic flat modules and $K(R\mathrm{-Proj})$, submitted.


\bibitem{iacob:15:gor.flat.proj}
S. Estrada, A. Iacob and S. Odaba\c{s}\i.
\newblock {Gorenstein flat and projective (pre)covers}.
\newblock submitted. arXiv:1508.04173.






\bibitem{fossum:75:auslander}
R. M. Fossum, P.A. Griffith and I. Reiten.
\newblock Trivial Extensions of Abelian Categories.
\newblock { \em Lecture Notes in Mathematics}, 456, Springer-Verlag, Berlin, 1975.

\bibitem{Gill04}
J. Gillespie. \newblock The flat model structure on Ch(R).\newblock {\em Trans.
Amer. Math. Soc}, 356(8):3369--3390, 2004..

\bibitem{holm:05:gor.dim}
H. Holm.
\newblock Gorenstein homological dimensions.
\newblock { \em J. Pure Appl. Algebra}, 189(1): 167--193, 2004.


\bibitem{huang:14:auslander}
Z. Huang.
\newblock On Auslander type conditions of modules.
\newblock preprint (available at http://maths.nju.edu.cn/~huangzy/)


\bibitem{iacob:15:gor.inj}
 A. Iacob.
\newblock {Gorenstein injective envelopes and covers over two sided noetherian rings}.
\newblock { \em Comm. Algebra}, to appear.



\bibitem{iacob:15:gor.flat}
 A. Iacob.
\newblock {Gorenstein flat preenvelopes}.
\newblock { \em Osaka J. Math.}, 52: 895--903, 2015.

\bibitem{IyengarKrause}
S. Iyengar and H. Krause.
\newblock{Acyclicity versus total acyclicity for complexes over noetherian rings.} \newblock{\em Doc. Math.}, 11:207--240, 2006.


\bibitem{iwanaga:79:gor}
Y. Iwanaga.
\newblock On rings with finite self injective dimension.
\newblock { \em Comm. Algebra}, 7(4): 393--414, 1979.

\bibitem{iwanaga:80:gor}
Y. Iwanaga.
\newblock On rings with finite self injective dimension II.
\newblock { \em Tsukuba J. Math.}, 4(1): 107--113, 1980.

\bibitem{Krause}
 H. Krause. \newblock{The stable derived category of a {N}oetherian
                   scheme}. \newblock{\em Compos. Math.}, 141(5): 1128--1162, 2005.



\bibitem{murfet:11:gor.proj}
D. Murfet and S. Salarian.
\newblock {Totally acyclic complexes over noetherian schemes}.
\newblock {\em {Adv. Math.}}, 226:1096--1133, 2011.

\bibitem{neeman}
A. Neeman.
\newblock { The homotopy category of flat modules, and Grothendieck duality}.
\newblock {\em {Inv. Math.}}, 174(2):255-308, 2008.
\bibitem{yan:10:cotorsion}
H. Yan.
\newblock {Strongly cotorsion (torsion-free) modules and cotorsion pairs}.
\newblock {\em Bull. Korean Math. Soc.}, 47:1041--1052, 2010.

\bibitem{yang:11:cotorsion}
G. Yang and Z. Liu.
\newblock {Cotorsion pairs and model structures on Ch(R)}.
\newblock {\em {Proc. Edinburgh Math. Soc.}}, (2) 52:783--797, 2011.

\bibitem{yang:14:gorflat}
G. Yang.
\newblock {All modules have Gorenstein flat precovers}.
\newblock {\em {Comm. Algebra}}, 42(7): 3078--3085, 2014.


\end{thebibliography}

\end{document}